\documentclass[12pt,a4paper,oneside]{amsart}
\usepackage{fullpage}
\usepackage{amssymb,amsfonts}
\usepackage{epsfig}
\usepackage{graphicx}

\usepackage{graphics}
\newtheorem{theorem}{Theorem}[section]
\newtheorem{lemma}[theorem]{Lemma}

\newtheorem{proposition}[theorem]{Proposition}

\theoremstyle{definition}
\newtheorem{definition}[theorem]{Definition}
\newtheorem{remark}[theorem]{Remark}

\numberwithin{equation}{section}

\DeclareMathOperator{\diam}{diam} 

\newcommand{\be}{\begin{equation}}
\newcommand{\ee}{\end{equation}}

\DeclareMathOperator{\supp}{supp}

\makeindex

\def\Xint#1{\mathchoice 
 {\XXint\displaystyle\textstyle{#1}}%
{\XXint\textstyle\scriptstyle{#1}}%
{\XXint\scriptstyle\scriptscriptstyle{#1}}%
 {\XXint\scriptscriptstyle\scriptscriptstyle{#1}}%
 \!\int}
\def\XXint#1#2#3{{\setbox0=\hbox{$#1{#2#3}{\int}$}
 \vcenter{\hbox{$#2#3$}}\kern-.5\wd0}}

 \def\dashint{\Xint-}

\usepackage[colorinlistoftodos]{todonotes}
\usepackage[colorlinks=true, allcolors=blue]{hyperref}

\title{Lower bound of measure and embeddings of Sobolev, Besov and Triebel-Lizorkin spaces}
\author{Nijjwal Karak}
\address{Discipline of Mathematics, Indian Institute of Technology Indore, Simrol, Indore-453552, India}
\email{nijjwal@gmail.com}
\thanks{This work was supported by DST-SERB (Grant no. PDF/2016/000328).}
\begin{document}

\begin{abstract}
In this article, we study the relation between Sobolev-type embeddings for Sobolev spaces or Haj\l asz-Besov spaces or Haj\l asz-Triebel-Lizorkin spaces defined either on a doubling or on a geodesic metric measure space and lower bound for measure of balls either in the whole space or in a domain inside the space.  
\end{abstract}
\maketitle
\section{Introduction}
This paper concerns with the necessity condition for Sobolev-type embedding for Haj\l asz-Sobolev spaces $M^{s,p}(X)$ as well as Haj\l asz-Besov spaces $N^s_{p,q}(X)$ and Haj\l asz-Triebel-Lizorkin spaces $M^s_{p,q}(X)$ defined on a metric measure space $(X,d,\mu);$ see Section 2 for the definitions of these function spaces. In particular, we ask if the Sobolev-type embedding holds then whether the measure of the balls has some lower bound or whether the domain satisfies the so-called measure density condition. Note that a domain $\Omega$ in a $Q$-doubling metric measure space $(X,d,\mu),$ $Q>1,$ satisfies the measure density condition if there exists a constant $c>0$ such that for all $x\in\Omega$ and for all $0<r\leq 1,$
\begin{equation}\label{measuredensity}
\mu(B(x,r)\cap\Omega)\geq cr^Q.
\end{equation} 
In \cite{HKT08b}, it has been proved that if the Sobolev embedding theorem holds in $\Omega\subset\mathbb{R}^n,$ in any of all the possible cases, then $\Omega$ satisfies the measure density condition. One can also look at \cite{HKT08} and \cite{Zho15} for more results in this direction. Similar results for Haj\l asz-Besov spaces and Haj\l asz-Trieble-Lizorkin spaces have been studied in \cite{HIT16} and \cite{K}. Recently, G\'orka in \cite{Gor17} has proved that in a metric measure space $(X,d,\mu),$ where $\mu$ is a doubling 
measure, if the embedding $M^{1,p}(X)\hookrightarrow L^{p^*}(X)$ holds for any $p^*>p,$ then there exists $b>0$ such that for any $x\in X$ and $0<r\leq 1,$ the following inequality holds
\begin{equation*}
\mu(B(x,r))\geq br^Q,
\end{equation*}
where $\frac{1}{p}-\frac{1}{p^*}=\frac{1}{Q}.$ Here and in what follows the symbol $\hookrightarrow$ denotes a continuous embedding. In this paper, we prove that in a geodesic space $(X,d,\mu),$ if the embedding $M^{1,p}(\Omega)\hookrightarrow L^{p^*}(\Omega)$ holds for a domain $\Omega\subset X,$ then $\Omega$ satisfies the measure density condition \eqref{measuredensity}. We have also considered the cases when $p>Q$ and $p=Q$ for $M^{1,p}(X)$ in a $Q$-doubling space. We have also established the same result as in \cite{Gor17} for Haj\l asz-Besov and Haj\l asz-Triebel-Lizorkin spaces.
\section{preliminaries}
\indent Let $(X,d,\mu)$ be a metric measure space equipped with a metric $d$ and a Borel regular measure $\mu.$ We assume throughout the paper that the measure of every open nonempty set is positive and that of every bounded set is finite.
\indent The measure $\mu$ is called \textit{doubling} if there exist a constant $C_d\geq 1$ such that
\begin{equation*}
\mu(B(x,2r))\leq C_d\,\mu(B(x,r))
\end{equation*}
for each $x\in X$ and $r>0.$ We call a triple $(X,d,\mu)$ a \textit{doubling metric measure space} if $\mu$ is a doubling measure on $X.$ If the measure $\mu$ is doubling, then a simple iteration argument (see \cite{HKST15}) shows that there
is an exponent $Q>0$ and a constant $C\geq 1$ so that
\begin{equation}\label{Qdoubling}
\left(\frac{s}{r}\right)^Q\leq C\frac{\mu(B(x,s))}{\mu(B(a,r))}
\end{equation}
holds whenever $a\in X$, $x\in B(a,r)$ and $0<s\leq r.$ We say that $(X,d,\mu)$ is $Q$-\textit{doubling} if $(X,d,\mu)$ is a doubling metric measure space and \eqref{Qdoubling} holds with the given $Q.$\\
\indent A metric space $X$ is said to be \textit{geodesic} if every pair of points in the space can be joined by a curve whose length is equal to the distance between the points.\\

Besov and Triebel-Lizorkin spaces are certain generalizations of fractional Sobolev spaces. There are several ways to define these spaces in the Euclidean setting and also in the metric setting. For various definitions of in the metric setting, see \cite{GKS10}, \cite{GKZ13}, \cite{KYZ11} and the references therein. In this paper, we use the approach based on pointwise inequalities, introduced in \cite{KYZ11}.
\begin{definition}
Let $S\subset X$ be a measurable set and let $0<s<\infty.$ A sequence of nonnegative measurable functions $(g_k)_{k\in\mathbb{Z}}$ is a fractional $s$-gradient of a measurable function $u:S\rightarrow [-\infty,\infty]$ in $S,$ if there exists a set $E\subset S$ with $\mu(E)=0$ such that
\begin{equation}\label{Hajlasz}
\vert u(x)-u(y)\vert\leq d(x,y)^s\left(g_k(x)+g_k(y)\right)
\end{equation} 
for all $k\in\mathbb{Z}$ and for all $x,y\in S\setminus E$ satisfying $2^{k-1}\leq d(x,y)<2^{k}.$ The collection of all fractional $s$-gradient of $u$ is denoted by $\mathbb{D}^s(u).$
\end{definition}
Let $S\subset X$ be a measurable set. For $0<p,q\leq \infty$ and a sequence $\vec{f}=(f_k)_{k\in\mathbb{Z}}$ of measurable functions, we define
\begin{equation*}
\Vert (f_k)_{k\in\mathbb{Z}}\Vert_{L^p(S,l^q)}=\big\Vert \Vert (f_k)_{k\in\mathbb{Z}}\Vert_{l^q} \big\Vert_{L^p(S)}
\end{equation*}
and
\begin{equation*}
\Vert (f_k)_{k\in\mathbb{Z}}\Vert_{l^q(L^p(S))}=\big\Vert (\Vert (f_k)\Vert_{L^p(S)})_{k\in\mathbb{Z}} \big\Vert_{l^q},
\end{equation*}
where
\begin{equation*}
\Vert (f_k)_{k\in\mathbb{Z}}\Vert_{l^q}=
\begin{cases}
(\sum_{k\in\mathbb{Z}}\vert f_k\vert^q)^{1/q},& ~\text{when}~0<q<\infty,\\
\sup_{k\in\mathbb{Z}}\vert f_k\vert,& ~\text{when}~q=\infty.
\end{cases}
\end{equation*}
\begin{definition}
Let $S\subset X$ be a measurable set. Let $0<s<\infty$ and $0<p,q\leq\infty.$ The \textit{homogeneous Haj\l asz-Triebel-Lizorkin space} $\dot{M}^s_{p,q}(S)$ consists of measurable functions $u:S\rightarrow [-\infty,\infty],$ for which the (semi)norm
\begin{equation*}
\Vert u\Vert_{\dot{M}^s_{p,q}(S)}=\inf_{\vec{g}\in\mathbb{D}^s(u)}\Vert\vec{g}\Vert_{L^p(S,l^q)}
\end{equation*}
is finite. The (non-homogeneous) \textit{Haj\l asz-Triebel-Lizorkin space} $M^s_{p,q}(S)$ is $\dot{M}^s_{p,q}(S)\cap L^p(S)$ equipped with the norm
\begin{equation*}
\Vert u\Vert_{M^s_{p,q}(S)}=\Vert u\Vert_{L^p(S)}+\Vert u\Vert_{\dot{M}^s_{p,q}(S)}.
\end{equation*} 
Similarly, the \textit{homogeneous Haj\l asz-Besov space} $\dot{N}^s_{p,q}(S)$ consists of measurable functions $u:S\rightarrow [-\infty,\infty],$ for which the (semi)norm
\begin{equation*}
\Vert u\Vert_{\dot{N}^s_{p,q}(S)}=\inf_{\vec{g}\in\mathbb{D}^s(u)}\Vert\vec{g}\Vert_{l^q(L^p(S))}
\end{equation*}
is finite and the (non-homogeneous) \textit{Haj\l asz-Besov space} $N^s_{p,q}(S)$ is $\dot{N}^s_{p,q}(S)\cap L^p(S)$ equipped with the norm
\begin{equation*}
\Vert u\Vert_{N^s_{p,q}(S)}=\Vert u\Vert_{L^p(S)}+\Vert u\Vert_{\dot{N}^s_{p,q}(S)}.
\end{equation*} 
\end{definition}
\noindent The space $M^s_{p,q}(\mathbb{R}^n)$ given by the metric definition coincides with the Triebel-Lizorkin space $F^s_{p,q}(\mathbb{R}^n),$ defined via the Fourier analytic approach, when $0<s<1,$ $n/(n+s)<p<\infty$ and $0<q\leq \infty,$ see \cite{KYZ11}. Similarly, $N^s_{p,q}(\mathbb{R}^n)$ coincides with Besov space $B^s_{p,q}(\mathbb{R}^n)$ for $0<s<1,$ $n/(n+s)<p<\infty$ and $0<q\leq \infty,$ see \cite{KYZ11}. For the definitions of $F^s_{p,q}(\mathbb{R}^n)$ and $B^s_{p,q}(\mathbb{R}^n),$ we refer to \cite{Tri83} and \cite{Tri92}.\\
Let us also recall the definition of Haj\l asz-Sobolev space $M^{s,p}(S),$ which is due to Haj\l asz for $s=1,$ \cite{Haj96} and to Yang for fractional scales, \cite{Yan03}.
\begin{definition}
Let $S\subset X$ be a measurable set. Let $0<s<\infty$ and $0<p\leq\infty.$ A nonnegative measurable function $g$ is an $s$-gradient of a measurable function $u$ in $S$ if there exists a set $E\subset S$ with $\mu(E)=0$ such that for all $x,y\in S\setminus E,$
\begin{equation*}
\vert u(x)-u(y)\vert\leq d(x,y)^s(g(x)+g(y)).
\end{equation*} 
The collection of all $s$-gradients of $u$ is denoted by $\mathcal{D}^s(u).$ The \textit{homogeneous Haj\l asz-Sobolev space} $\dot{M}^{s,p}(S)$ consists of measurable functions $u$ for which 
\begin{equation*}
\Vert u\Vert_{\dot{M}^{s,p}(S)}=\inf_{g\in\mathcal{D}^s(u)}\Vert g\Vert_{L^p(S)}
\end{equation*}
is finite. The \textit{Haj\l asz-Sobolev space} $M^{s,p}(S)$ is $\dot{M}^{s,p}(S)\cap L^p(S)$ equipped with the norm
\begin{equation*}
\Vert u\Vert_{M^{s,p}(S)}=\Vert u\Vert_{L^p(S)}+\Vert u\Vert_{\dot{M}^{s,p}(S)}.
\end{equation*}
\end{definition}
\noindent Note that if $0<s<\infty$ and $0<p\leq\infty,$ then $\dot{M}^s_{p,\infty}(X)=\dot{M}^{s,p}(X),$ \cite[Proposition 2.1]{KYZ11}.\\

\section*{Main Results}
Our first result gives the measure density condition from Sobolev embedding in a geodesic metric measure space. In particular, when $\Omega=X,$ it gives the result of \cite{Gor17} without the doubling but the geodesity assumption. Note that this result was proved in \cite{K} for $M^s_{p,q}(X)$ and $N^s_{p,q}(X),$ where $0<s<1, 0<p<\infty, 0<q\leq\infty$ and hence in particular, for $M^{s,p}(X),$ where $0<s<1, 0<p<\infty.$
\begin{theorem}\label{fordomain}
Suppose $(X,d,\mu)$ is a geodesic metric measure space. If for any domain $\Omega\subset X,$ $M^{1,p}(\Omega)\hookrightarrow L^{p^*}(\Omega),$ where $p^*>p\geq 1,$ then there exists a constant $b$ depending on $p,p^*$ and the embedding constant such that
\begin{equation*}
\mu(B(x,R)\cap\Omega)\geq bR^Q
\end{equation*}
holds for all $R\in (0,1],$ where $\frac{1}{p}-\frac{1}{p^*}=\frac{1}{Q}.$
\end{theorem}
\begin{proof}
For $x\in X$ and $R>0,$ let $y\in B(x,R).$ Take $\tilde{R}<R.$ Since $y$ can be connected to $x$ by a geodesic, it follows that the set $B(y,\tilde{R})\cap B(x,R)$ contains a ball of radius $\tilde{R}/2$ and hence
\begin{equation*}
 \mu(B(y,\tilde{R})\cap B(x,R))\geq C\mu(B(y,\tilde{R})),
 \end{equation*}
 thanks to \eqref{Qdoubling}. This and Lebesgue differentiation theorem implies that $\mu(\partial B(x,R))=0.$ Since this is true for all $R>0,$ so there always exists $r<R$ such that
\begin{equation}\label{applicationofgeodesic}
\mu(B(x,r)\cap\Omega)=\frac{1}{2}\mu(B(x,R)\cap\Omega).
\end{equation}
Let $u:\Omega\rightarrow [0,1]$ be defined as
\begin{equation}\label{mainfn}
u(y)=
\begin{cases}
  1& \text{if $y\in B(x,r)\cap\Omega$},\\
  \frac{R-d(x,y)}{R-r} & \text{if $y\in B(x,R)\setminus B(x,r)\cap\Omega$},\\
  0& \text{if $y\in \Omega\setminus B(x,R)$}.
 \end{cases}
\end{equation}
Note that $\Vert u\Vert_{L^{p^*}(\Omega)}\geq \mu(B(x,r)\cap\Omega)^{\frac{1}{p^*}}.$ Also note that $g(y)=\frac{1}{R-r}\chi_{B(x,R)}$ is a generalized gradient of $u$ and hence
\begin{eqnarray*}
\Vert u\Vert_{M^{1,p}(\Omega)} &\leq & \left(\int_{\Omega}\vert u\vert^p\,d\mu \right)^{\frac{1}{p}}+\left(\int_{\Omega}\vert g\vert^p\, d\mu \right)^{\frac{1}{p}}\\
&\leq & \mu(B(x,R)\cap\Omega)^{\frac{1}{p}}+\frac{1}{R-r}\mu(B(x,R)\cap\Omega)^{\frac{1}{p}}\\
&\leq & \frac{2}{R-r}\mu(B(x,R)\cap\Omega)^{\frac{1}{p}},
\end{eqnarray*}
as $r<R<1.$ Therefore, from hypothesis, we obtain
\begin{equation*}
R-r\leq C\mu(B(x,R)\cap\Omega)^{\frac{1}{Q}}
\end{equation*}
Now let us define a sequence $r_0>r_1>r_2>\cdots>0$ by induction as follows: $r_0=R,$ $r_1=r$ and $r_2$ is the radius comes from \eqref{applicationofgeodesic} for $r_1$ and so on.
 Clearly,
\begin{equation}\label{decreasing}
\mu(B(x,r_j)\cap\Omega)=2^{-j}\mu(B(x,R)\cap\Omega).
\end{equation}
Therefore $r_j\rightarrow 0$ as $j\rightarrow\infty,$ and hence
\begin{eqnarray*}
R &=&\sum_{j=0}^{\infty}(r_j-r_{j+1})\\
&\lesssim & \sum_{j=0}^{\infty}2^{-j/Q}\mu(B(x,R)\cap\Omega)^{1/Q}\\
&\leq & \mu(B(x,R)\cap\Omega)^{1/Q},
\end{eqnarray*}
as desired.
\end{proof}
In a $Q$-doubling space, $Q>1,$ Theorem \ref{fordomain} considers the case when $p<Q$ for $M^{1,p}(X).$ Now we will consider the other two cases, i.e., when $p=Q$ and $p>Q.$
\begin{theorem}\label{connected}
Suppose that $(X,d,\mu)$ is a connected $Q$-doubling space, $Q>1.$\\
(i) If there exist constants $C_1, C_2$ such that for all $u\in M^{1,Q}(X)$ and for all balls $B\in X,$ we have
\begin{equation}\label{embedding}
\dashint_{B}\exp\left(C_1\frac{\vert u-u_B\vert}{\Vert u\Vert_{M^{1,Q}}}\right)\,d\mu\leq C_2,
\end{equation}
then there exists $b=b(C_1,C_2,C_d)$ such that $\mu(B(x,r))\geq br^Q$ holds for all $x\in X$ and for all $0<r<\min\{1,(\diam X)/2\}.$\\
(ii) If there exists a constant $C_3$ such that for all $u\in M^{1,p}(X),$ $p>Q,$ and for every $y,z\in X,$ we have $\vert u(y)-u(z)\vert\leq C_3\Vert u\Vert_{M^{1,p}(X)}d(y,z)^{1-Q/p},$ then there exists $b=b(C_3,C_d)$ such that $\mu(B(x,r))\geq br^Q$ holds for all $x\in X$ and for all $0<r<\min\{1,(\diam X)/2\}.$
\end{theorem}
\begin{proof}
(i) For a fixed $x\in X$ and $r>0,$ let us define
\begin{equation}\label{mainfunction}
u(y)=
\begin{cases}
  1& \text{if $y\in B(x,\frac{r}{2})$},\\
  \frac{2}{r}(r-d(x,y))& \text{if $y\in B(x,r)\setminus B(x,\frac{r}{2})$},\\
  0& \text{if $y\in X\setminus B(x,r)$}.
 \end{cases}
\end{equation}
It is easily seen that the function $g(y)=\frac{2}{r}\chi_{B(x,r)}$ is a $1$-gradient of $u$ and we have the estimate $\Vert u\Vert_{M^{1,Q}(X)}\leq \frac{3}{r}(\mu(B(x,r)))^{\frac{1}{Q}}$ for all $r\leq 1.$ Hence, by hypothesis, we obtain
\begin{equation*}
\int_{B}\exp\left(\frac{C_1r\vert u-u_B\vert}{\mu(B(x,r))^{1/Q}}\right)\,d\mu\leq C_2\mu(B(x,r))
\end{equation*}
Since $u=1$ on $B(x,\frac{r}{2})$ and $u=0$ on $X\setminus B(x,r),$ we have that $\vert u-u_{B(x,r)}\vert\geq\frac{1}{2}$ on at least one of the sets $B(x,\frac{r}{2})$ and $X\setminus B(x,r).$ Hence,
\begin{equation}\label{maininequality}
\min\{\mu(B(x,\frac{r}{2})),\mu(X\setminus B(x,r))\}\exp\left(\frac{C_1r}{\mu(B(x,r))^{1/Q}}\right)\leq C_2\mu(B(x,r)).
\end{equation}
Since $r<(\diam X)/2,$ there is a point $z\in X\setminus B(x,2r)$ and hence $\partial B(x,3r/2) $ is nonempty as $X$ is connected. Therefore we can find a point $w\in B(x,2r)$ such that $d(x,w)=3r/2$ and $B(w,r/2)\subset B(x,2r)\setminus B(x,r).$ Owing to \eqref{Qdoubling}, we ensure that there exists a constant $C=C(C_d)$ such that
\begin{equation*}
\mu(B(x,r))\leq\mu(B(x,2r))\leq C\mu(B(w,\frac{r}{2}))\leq C\mu(B(x,2r)\setminus B(x,r))
\end{equation*}
holds for all $x\in X$ and for all $0<r<(\diam X)/2.$ This implies that there exists a constant $C'(C_d)$ such that $\min\{\mu(B(x,\frac{r}{2})),\mu(X\setminus B(x,r))\}\geq C'\mu(B(x,r)).$ Therefore the conclusion follows from \eqref{maininequality}.\\
(ii) For $x\in X,$ define the function $u$ as in \eqref{mainfunction} for $x,r,\frac{r}{4}$ instead of $x,r,\frac{r}{2}.$ Again use the estimate $\Vert u\Vert_{M^{1,p}}\lesssim \frac{1}{r}\mu(B(x,r))^{\frac{1}{p}},$ $r\leq 1,$ in the hypothesis to obtain
\begin{equation*}
\vert u(y)-u(z)\vert\lesssim \frac{\mu(B(x,r))^{1/p}}{r} d(y,z)^{1-Q/p},
\end{equation*}
which is true for all $y,z\in X.$ In particular, let $y\in B(x,\frac{r}{4})$ and $z\in B(x,\frac{3r}{2})\setminus B(x,r).$ Such points exist because of the connectivity of $X$ and the choice of $r.$ Therefore $u(y)=1,$ $u(z)=0$ and $d(y,z)\leq 2r.$ Hence $\mu(B(x,r))\gtrsim r^Q.$  
\end{proof}
If we assume, in addition, that the space $X$ is geodesic, then one gets slightly better conclusion, see Theorem \ref{withoutproof}, using a similar proof as of Theorem 5.1 of \cite{K}; one needs to use the estimate of $\Vert u\Vert_{M^{1,p}}$ as in the above theorem and the rest will be same as the proof of Theorem 5.1 of \cite{K}. We will omit the proof here.
\begin{theorem}\label{withoutproof}
Suppose that $(X,d,\mu)$ is a geodesic space and let $\Omega$ be a domain in $X.$\\
(i) If there exist constants $C_1, C_2$ such that for all $u\in M^{1,Q}(\Omega),$ $Q>1,$ and for all balls $B_R\in X,$ we have
\begin{equation*}
\int_{B_R\cap\Omega}\exp\left(C_1\frac{\vert u-u_B\vert}{\Vert u\Vert_{M^{1,Q}}(\Omega)}\right)\,d\mu\leq C_2R^Q,
\end{equation*}
then there exists $b=b(C_1,C_2,C_d)$ such that $\mu(B(x,r)\cap\Omega)\geq br^Q$ holds for all $x\in X$ and for all $0<r\leq 1.$\\
(ii) If there exists a constant $C_3$ such that for all $u\in M^{1,p}(\Omega),$ and for every $y,z\in \Omega,$ we have $\vert u(y)-u(z)\vert\leq C_3\Vert u\Vert_{M^{1,p}(\Omega)}d(y,z)^{1-Q/p},$ for some $1<Q<p,$ then there exists $b=b(C_3,C_d)$ such that $\mu(B(x,r)\cap\Omega)\geq br^Q$ holds for all $x\in X$ and for all $0<r\leq 1.$
\end{theorem}
Now we will consider the spaces $M^s_{p,q}(X), N^s_{p,q}(X)$ and prove a similar result as in \cite{Gor17}. We will need the following lemma, which is proved using the idea of the proof of Lemma 3.10 of \cite{HIT16}.
\begin{lemma}\label{gradient}
Let $0<s<1,$ $0<p<\infty$ and $0<q\leq\infty$ and let $F\subset X$ be a measurable set, where $X=(X,d,\mu)$ is a metric space. Let $u$ be a bounded $L$-Lipschitz function supported in $F.$ Then the sequence $\vec{g}=(g_k)_{k\in\mathbb{Z}}$ defined by
\begin{equation*}
g_k=
 \begin{cases}
  2^{k(s-1)}L\chi_{\supp u} & \text{if $k\geq k_L$},\\
  2^{sk+2}\Vert u\Vert_{\infty}\chi_{\supp u}& \text{if $k<k_L$},
 \end{cases}
\end{equation*}
is a fractional $s$-gradient of $u,$ where $k_L$ is an integer such that $2^{k_L-1}<L\leq 2^{k_L}.$ Moreover, $\Vert \vec{g}\Vert_{L^p(X,l^q)}\leq C_1(1+\Vert u\Vert_{\infty})L^s\,\mu(F)^{\frac{1}{p}}$ and $\Vert \vec{g}\Vert_{l^q(L^p)(X)}\leq C_2(1+\Vert u\Vert_{\infty})L^s\,\mu(F)^{\frac{1}{p}}$ for some constants $C_1, C_2$ depending on $s$ and $q.$
\end{lemma}
\begin{proof}
Let $x,y\in X$ and let $k\in\mathbb{Z}$ such that $2^{-k-1}\leq d(x,y)<2^{-k}.$ If $x,y\in \supp u,$ then
\begin{eqnarray*}
\vert u(x)-u(y)\leq Ld(x,y)<d(x,y)^s\,L\,2^{k(s-1)}\leq d(x,y)^s(\rho_k(x)+\rho_k(y)),
\end{eqnarray*}
where $\rho_k=2^{k(s-1)}L\chi_{\supp u}$ for all $k\in\mathbb{Z}.$ The other cases when one of $x$ and $y$ is in $\supp u$ or none of them is in $\supp u$ can considered similarly. Again, for $x,y\in\supp u,$ we have
\begin{eqnarray*}
\vert u(x)-u(y)\vert\leq 2\Vert u\Vert_{\infty}\leq d(x,y)^s2^{sk+2}\Vert u\Vert_{\infty}\leq d(x,y)^s(h_k(x)+h_k(y)),
\end{eqnarray*} 
where $h_k=2^{sk+2}\Vert u\Vert_{\infty}\chi_{\supp u}$ for all $k\in\mathbb{Z}.$ Therefore, $(g_k)_{k\in\mathbb{Z}}$ is also a fractional $s$-gradient of $u.$ Moreover, we have
\begin{eqnarray*}
\left(\sum_{k\in\mathbb{Z}}\vert g_k\vert ^q\right)^{1/q}\leq C_1'\left[\Vert u\Vert_{\infty}\left(\sum_{k<k_L}2^{(sk+2)q}\right)^{1/q}+L\left(\sum_{k\geq k_L}2^{kq(s-1)}\right)^{1/q}\right]\chi_{\supp u}
\end{eqnarray*}
and hence
\begin{eqnarray*}
\Vert \vec{g}\Vert_{L^p(X,l^q)} &\leq & C_1''(\Vert u\Vert_{\infty}2^{sk_L}+L2^{k_L(s-1)})\mu(F)^{1/p}\\
&\leq & C_1L^s(1+\Vert u\Vert_{\infty})\mu(F)^{1/p}.
\end{eqnarray*}
Similarly, 
\begin{eqnarray*}
\Vert \vec{g}\Vert_{l^q(L^p)(X)}\leq\left(\sum_{k\in\mathbb{Z}}\Vert g_k\Vert_{L^p(X)}^q\right)^{\frac{1}{q}} &\leq & C_2'\mu(F)^{\frac{1}{p}}\left(\Vert u\Vert_{\infty}\left(\sum_{k<k_L}2^{(sk+2)q}\right)^{\frac{1}{q}}+\left(\sum_{k<k_L}2^{k(s-1)q}\right)^{\frac{1}{q}}\right)\\
&\leq & C_2''\mu(F)^{\frac{1}{p}}\left(\Vert u\Vert_{\infty}2^{sk_L}+L2^{k_L(s-1)}\right)\\
&\leq & C_2L^s(1+\Vert u\Vert_{\infty})\mu(F)^{1/p}.
\end{eqnarray*}
\end{proof}
\begin{theorem}
Let $0<s<1,$ $0<p<\infty$ and $0<q\leq\infty.$ Suppose that $(X,d,\mu)$ is a metric measure space with the doubling measure $\mu.$ If 
\begin{equation*}
M^s_{p,q}(X)\hookrightarrow L^{p^*}(X),
\end{equation*}
where $p^*>p,$ then there exists $C=C(s,p,q,p^*,C_e)$ such that
\begin{equation*}
\mu(B(x,r))\geq Cr^Q,\hspace{2cm} \text{for}\quad r\in (0,1],
\end{equation*}
where $\frac{1}{p}-\frac{1}{p^*}=\frac{s}{Q}$ and $C_e$ is the constant of the embedding. The claim also holds with $M^s_{p,q}(X)$ replaced by $N^s_{p,q}(X).$
\end{theorem}
\begin{proof}
For each $u\in M^s_{p,q}(X),$ we have
\begin{equation}\label{first}
\Vert u\Vert_{L^{p^*}(X)}\leq C_e\left(\Vert u\Vert_{L^p(X)}+\Vert\vec{g}\Vert_{L^p(X, l^q)}\right),
\end{equation}
where $\vec{g}=(g_k)_{k\in\mathbb{Z}}$ is a fractional $s$-gradient of $u.$ For a fixed $x\in X$ and for $0<R,$ let us define a function $u$ as follows
\begin{equation*}
u(y)=
\begin{cases}
  1& \text{if $y\in B(x,\frac{R}{2})\cap\Omega$},\\
  \frac{2}{R}(R-d(x,y)) & \text{if $y\in B(x,R)\setminus B(x,\frac{R}{2})\cap\Omega$},\\
  0& \text{if $y\in \Omega\setminus B(x,R)$}.
 \end{cases}
\end{equation*}
Note that $u$ is a $\frac{2}{R}$-Lipschitz function and hence by Lemma \ref{gradient}, the sequence $\{g_k\}_{k\in\mathbb{Z}}$ defined by
\begin{equation*}
g_k=
 \begin{cases}
  2^{k(s-1)}\,\frac{2}{R}\chi_{B(x,R)} & \text{if $k\geq k_L$},\\
  2^{sk+2}\chi_{B(x,R)}& \text{if $k<k_L$},
 \end{cases}
\end{equation*}
is a fractional $s$-gradient of $u$ with $\Vert \vec{g}\Vert_{L^p(X,l^q)}\leq C_g(2/R)^s\,\mu(B(x,R))^{\frac{1}{p}}.$ Therefore, from \eqref{first} and H\"older inequality, we have
\begin{equation}\label{holderapplication2}
\frac{1}{\mu(B(x,R))^{\frac{s}{Q}}}-C_e\leq C_e\frac{\Vert \vec{g}\Vert_{L^p(X,l^q)}}{\Vert u\Vert_{L^p(X)}}.
\end{equation}
Let us fix $r\leq 1.$ If we have $\mu(B(x,r)\cap\Omega)\geq \left(\frac{1}{2C_e}\right)^Qr^Q,$ then we are done. Thus we may assume that $\mu(B(x,r)\cap\Omega)\leq \left(\frac{1}{2C_e}\right)^Qr^Q.$ Thus for any $\delta\leq r,$ we have $\mu(B(x,\delta)\cap\Omega)\leq \left(\frac{1}{2C_e}\right)^Q.$ It implies, together with \eqref{holderapplication2}, that 
\begin{equation*}
\frac{1}{(2C_e)^p}\mu(B(x,\delta))^{-\frac{sp}{Q}}\leq \frac{\Vert \vec{g}\Vert_{L^p(X,l^q)}^p}{\Vert u\Vert_{L^p(X)}^p}.
\end{equation*}
Consequently, the structure of $u$ and $\vec{g}$ yields
\begin{equation*}
\frac{1}{(2C_e)^p}\mu(B(x,\delta))^{-\frac{sp}{Q}}\leq \frac{C_g^p(2/\delta)^{sp}\mu(B(x,\delta))}{\mu(B(x,\delta/2))}.
\end{equation*}
Therefore, we obtain, for each $\delta\leq r,$
\begin{equation}\label{finalestimate2}
 \mu(B(x,\delta))\geq \left(\frac{1}{2C_eC_g}\right)^{\frac{pQ}{Q+sp}}\left(\frac{\delta}{2}\right)^{\frac{spQ}{Q+sp}}\left(\mu(B(x,\delta/2))\right)^{\frac{Q}{Q+sp}}.
 \end{equation}
 Using the doubling condition, we easily get $\mu(B(x,r))\geq Cr^Q,$ where the constant $C$ depends on $s, p, q, p^*, C_e$ and $C_d.$ We can get rid of the dependence on the doubling constant $C_d$ by the same method introduced in \cite{Gor17}. Indeed, after iteration we get for any integer $n,$
 \begin{equation}\label{iteration}
 \mu(B(x,r))\geq \left(\frac{r^s}{2C_eC_g}\right)^{p\sum_{j=1}^n(\frac{Q}{Q+sp})^j}\left(\frac{1}{2}\right)^{sp\sum_{j=1}^nj(\frac{Q}{Q+sp})^j}\left(\mu\left(B(x,r/2^n)\right)\right)^{(\frac{Q}{Q+sp})^n}.
 \end{equation}
 Since the measure $\mu$ is doubling, we have
 \begin{equation*}
 \left(\mu\left(B(x,r)\right)\right)^{(\frac{Q}{Q+sp})^n}\geq\left(\mu\left(B(x,r/2^n)\right)\right)^{(\frac{Q}{Q+sp})^n}\geq C_d^{-n(\frac{Q}{Q+sp})^n} \left(\mu\left(B(x,r)\right)\right)^{(\frac{Q}{Q+sp})^n}
 \end{equation*}
 and hence
 \begin{equation*}
 \lim_{n\rightarrow\infty}\left(\mu\left(B(x,r/2^n)\right)\right)^{(\frac{Q}{Q+sp})^n}=1.
 \end{equation*}
 In addition, we have that
 \begin{align*}
 \sum_{n=1}^{\infty}\left(\frac{Q}{Q+sp}\right)^n=\frac{Q}{sp}\quad\mbox{and}\quad\sum_{n=1}^{\infty}n\left(\frac{Q}{Q+sp}\right)^n=\frac{Q(Q+sp)}{s^2p^2}.
 \end{align*}
 As a consequence, \eqref{iteration} gives us the desired lower bound, by letting $n\rightarrow\infty,$
 \begin{equation*}
 \mu(B(x,r))\geq\frac{1}{(2C_eC_g)^{\frac{Q}{s}}2^{\frac{Q(Q+sp)}{sp}}}r^Q.
 \end{equation*}
\end{proof}  
\begin{remark}
If we assume that the space $(X,d,\mu)$ is a $Q$-doubling space, $Q>1,$ then the above theorem deals with the case $sp<Q$ for the function spaces $M^s_{p,q}(X)$ and $N^s_{p,q}(X).$ For the cases $sp=Q$ and $sp>Q,$ we can get the same result as Theorem \ref{connected} for $M^s_{p,q}(X)$ and $N^s_{p,q}(X)$ by mimicking the proof of the same and by using Corollary 3.12 of \cite{HIT16}. In that case we assume that the space is connected and $Q$-doubling; the same result in a geodesic space is due to \cite{K}. 
\end{remark}
Note that the result of \cite{Gor17} holds for $M^{1,p}(X),$ i.e., for $M^1_{p,\infty},$ $p\geq 1.$ The same result will follow for $N^1_{p,\infty}$ from the next proposition. However, we do not know whether such results are true when $q<\infty,$ of course, if the space does not support a $p$-Poincar\'e inequality as otherwise these spaces are trivial in this case, \cite{GKS10}.
\begin{proposition}
Let $0<s,p<\infty.$ Then $M^{s,p}(X)\subset N^s_{p,\infty}(X).$
\end{proposition}
\begin{proof}
Let $u\in M^{s,p}(X)$ and $g\in \mathcal{D}^s(u)$ be a $s$-gradient of $u.$ Taking $g_k\equiv g,$ we know that $\vec{g}=\{g_k\}_{k\in\mathbb{Z}}\in \mathbb{D}^s(u)$ and $\Vert \vec{g}\Vert_{l^{\infty}(L^p(X))}=\Vert g\Vert_{L^p(X)},$ which implies that $u\in N^s_{p,\infty}(X).$
\end{proof}
\def\bibname{References}
\bibliography{measure_density}
\bibliographystyle{alpha}

\end{document}